\documentclass{amsart}

\usepackage{amssymb}
\usepackage[utf8]{inputenc}
\usepackage[english]{babel}
\usepackage{amsmath}
\usepackage{color}
\usepackage{graphics}
\usepackage{epsfig}
\usepackage[all,cmtip]{xy}
\usepackage{pb-diagram,pb-xy}
\usepackage{bm}
\usepackage{multicol}
\usepackage[shortlabels]{enumitem}

\newtheorem{theorem}{Theorem}

\newtheorem{definition}{Definition}
\newtheorem{proposition}{Proposition}
\newtheorem{cor}{Corollary}
\newtheorem{remark}{Remark}

\def\proof{{\bf Proof.}}

\newcommand{\Z}{\mathbb{Z}}
\newcommand{\Q}{|Q|}
\newcommand{\f}{\longrightarrow}

\def\Aut{{\hbox{Aut}\,}}

\def\Id{{\mathop{\mbox{ Id}}}}

\begin{document}
\title{On invariant (co)homology of a group}

\author[Carlos Aquino, Rolando Jimenez, Martin Mijangos  and Quitzeh Morales]{Carlos Aquino, Rolando Jimenez, Martin Mijangos  and Quitzeh Morales Mel\'endez}

\date{}

\address{Instituto de Matem\'aticas, Unidad Oaxaca,
Universidad Nacional Aut\'onoma de M\'exico,  Le\'on 2, 68000 Oaxaca de Ju\'arez, Oaxaca,
M\'exico} \email{carlos.az2@gmail.com}

\address{Instituto de Matem\'aticas, Unidad Oaxaca,
Universidad Nacional Aut\'onoma de M\'exico,  Le\'on 2, 68000 Oaxaca de Ju\'arez, Oaxaca,
M\'exico} \email{rolando@matcuer.unam.mx}

\address{Instituto de Matem\'aticas, Unidad Oaxaca,
Universidad Nacional Aut\'onoma de M\'exico,  Le\'on 2, 68000 Oaxaca de Ju\'arez, Oaxaca,
M\'exico} \email{mitm3@hotmail.com}

\address{CONACYT -- Universidad Pedag\'ogica Nacional -- Unidad 201 Oaxaca,
Camino a la Zanjita S/N, Col. Noche Buena, Santa Cruz Xoxocotl\'an, Oaxaca.
C.P. 71230}\email{qmoralesme@conacyt.mx}



\let\thefootnote\relax\footnote{The first three authors were partially supported by SEP-CONACyT Grant 284621.}

\maketitle

\begin{abstract}
There are different notions of homology and cohomology that can be defined for a group with an action of another group by group automorphisms. In this paper we address three natural questions that arise in this context. Namely, the relation of these notions with the usual (co)homology of a semidirect product, the interpretation of the first homology group as some kind of abelianization and the classification of (invariant) group extensions.
\end{abstract}

\section*{Introduction}
In \cite{Knudson} Knudson defined homology and cohomology groups for a group $G$ with an action of another group $Q$ by group automorphisms and computed homology groups for some $\mathbb{Z}/2$ actions on cyclic groups. Later, in \cite{JLM} it was given a formula for the first homology group in terms of the action and a spectral sequence was constructed to compute homology in some cases. 

It is natural to ask in this context about the relation of Knudson groups with classical constructions and results on (co)homology of groups. In this paper, we will do the following:
\begin{itemize}
\item relate Knudson homology groups with the homology of the semidirect product $G\rtimes Q$ as an application of Hochschild-Serre  spectral sequence. This generalizes the spectral sequence in \cite{JLM};
\item give an interpretation of the first homology group as a so called \textit{weighed} abelianization, a suitable group made up by orbits of the action. The relation with classic group constructions is also adressed. And
\item define new cohomology groups and study the notion of invariant group extensions 
related to them. 
\end{itemize}
The cohomology groups defined coincide with those of Knudson in some very special cases. Some properties of the new groups are studied. The corresponding homology groups do not coincide in general with those of Knudson.

\section{Knudson groups, invariant resolutions and semidirect product}


Let $Q$ and $G$ be groups, and $Q\times G \to G $ an action of $Q$ on $G$ by 
group automorphisms, i.e.
$$q(g_1g_2)= q(g_1)q(g_2),\qquad q(1)= 1$$
 for every $ q\in Q, g_1,g_2 \in G$. In this case, we call $G$ a $Q$-group.
 Let $A$ be an abelian group with trivial $G$- and $Q$-actions and denote
by $C_*(G, A)$ the complex $C_*(G)\otimes A$ where $C_*(G)$ is the bar complex.
In this case, there is an induced action of $Q$ on $C_*(G, A)$ given by
\begin{equation}\label{accion-sobre-c(g)}
q([g_1|\cdots|g_n]\otimes a)=[qg_1|\cdots |qg_n]\otimes a.
\end{equation}
According to \cite{Knudson}, define the groups of homology and cohomology 
of invariant group chains as
\begin{align}
 H_*^Q(G, A)=&H_*(C(G,A)^Q)\\
 H_Q^*(G, A)=&H^*(Hom(C(G)^Q, A)).
\end{align}
Since the differential of the chain complex $C_*(G, A)$ is $Q$-equivariant,
there is a well defined action of $Q$ on the homology groups $H_*(G, A)$.
 The next theorem is proved in \cite{Knudson} when $Q$ is a finite group.
 
\begin{definition}
Let $A$ be an abelian group and $n\in \Z$ . We say that $n$ is invertible in 
$A$ if the group homomorphism $\psi:A\f A$ given by $\psi(a)=na$ for $a\in A$ is
an isomorphism.
\end{definition}
 
\begin{theorem}(\cite{Knudson} Proposition 3.3)\label{puntos-fijos-H}
 Let $G$ be a $Q$-group and let $A$ be an abelian group with trivial $G$- and 
 $Q$-actions. Assume that $\Q$ is invertible in $A$. Then the natural map
 $$i_*:H_*^Q(G, A)\f H_*(G, A)^Q$$ is an isomorphism.
\end{theorem}

The aim of this section is to prove the following:
\begin{theorem}Let $Q$ be a finite group. Let $G$ be a $Q$-group and let $A$ be an abelian group with trivial $G$- and 
 $Q$-actions. If $|Q|$ is invertible in $A$, then  
 $H_q^Q(G, A)\cong H_q(G\rtimes Q, A)$.
\end{theorem}

First we give some properties of the group homology in the case when the coefficients are invertible. Then we will use Theorem \ref{puntos-fijos-H} to relate the homology of invariant group chains to the homology of the semidirect product $G\rtimes Q$ through the Hochschild-Serre  spectral sequence. 

At the end of this section, we will show the usefulness of this theorem  through an example. Such an example is presented in \cite{JLM} but here is treated in a more simple way.

Let $Q$ be a finite group and $A$ be a $Q$-module. Consider the homomorphism 
 $N:A\f A$, such that 
 $N(a)=\sum_{q\in Q}qa$.
Since $N$ satisfies $N(qa)=N(a)$ and $im\, N\subseteq A^Q$, the homomorphism
 $\bar{N}:A_Q\f A^Q$ given by
         $\bar{N}(\bar{a})= \sum_{q\in Q}qa$
is well defined and is called the norm map.



\begin{proposition}\label{norma-iso}\label{inv-en-H}
 Let $G$ be a group and $A$ be a $G$-module. 
 \begin{enumerate}[a)]
 \item If $n\in \Z$ is invertible in $A$, then $n$ is invertible in $H_*(G, A)$.
  \item  If $G$ is finite and $|G|$ is invertible in $A$, then $\bar{N}:A_G\f A^G$ is an isomorphism.
 \end{enumerate}
\end{proposition}
\begin{proof}
a) follows from functoriality of $H_*(G, A)$ in the second coordinate. 
It is an exercise in \cite[$\S$3.1, p.59]{brown} to check that $|G|$ annihilates $ker\,\bar{N}$ and $coker\, \bar{N}$, which yields b) in the present situation.
$\hfill\square$
\end{proof}


From now on $G$ will denote an arbitrary  $Q$-group with $Q$ a finite group, 
$\varphi:Q\f \Aut(G)$ the group homomorphism inducing the action of $Q$ on $G$, 
and $A$ an abelian group with trivial  $G$- and $Q$-actions.

The $Q$-action on $G$ induces a short exact sequence of groups
$$1\f G\f G\rtimes_\varphi Q\f Q\f 1.$$
We will identify the subgroup 
$G\times \{e\}$ of $G\rtimes_\varphi Q$ with $G$ and $\{e\}\times Q$ with $Q$.
 There is an associated Hochschild-Serre spectral sequence
\begin{align} \label{suc-esp}
 E^2_{pq}=H_p(Q, H_q(G, A))\Rightarrow H_{p+q}(G\rtimes_\varphi  Q, A)
\end{align}
where the action of $Q$ on $H_q(G, A)$ is the one given in Corollary III.8.2 of \cite{brown}.
Recall that the action  of an element $q\in Q$ on $H_*(G, A)$ is given by the automorphism $c(q)_*:H_*(G,A)\f H_*(G, A)$
induced by the automorphism in the category of pairs (\cite{brown} III.8)
\begin{align*}
 c(q)=(\alpha_q, id_A):(G, A)&\f (G, A)\\
     (g,m )&\longmapsto (qgq^{-1}, m)
\end{align*}
where $\alpha_q:G\f G$ is given by $g\longmapsto qgq^{-1}$ and $id_A$ is the 
identity in $A$. Note in particular that $\varphi(q)=\alpha_q$. Given a projective resolution 
$F\stackrel{\epsilon}{\f}\Z$ on $\Z G$ in order to compute 
this action at the chain level it is necessary to find an augmentation 
preserving chain map of $G$-modules $\tau_q:F\f F$ such that 
$\tau_q(gx)=\alpha_q(g)\tau_q(x)$ for $ g\in G$ and $x\in F$. In this way, we will have
$c(q)_*=(\tau_q\otimes_G id_A)_*$. 

Consider the bar resolution 
$B_*(G)\stackrel{\epsilon}{\f}\Z$ and
\begin{align*}
 (\tau_q)_n:B_n(G)&\f B_n(G)\\
       (g_0,g_1,...,g_n)&\longmapsto (qg_0, qg_1,...,qg_n)
\end{align*}
for all $n\geq0$ where $qg_i$ represents the element $\varphi(q)(g_i)$. 
We can see that $\tau_q$ is an augmentation preserving map. Moreover,
given $g\in G$
\begin{align*}
 \tau_q(g(g_0,...,g_n))&=\tau_q(gg_0,...gg_n)&\mbox{for the action $G$ 
 on $B_n(G)$}\\
                       &=(q(gg_0),...,q(gg_n))&\\
                       &=(qgqg_0,...qgqg_n)&\mbox{since $q$ is 
                       automorphism of $G$}\\
                       &=qg(qg_0,...,qg_n)&\\
                       &=\alpha_q(g)\tau_q(g_0,...,g_n).&
\end{align*}
Hence $c(q)_*=(\tau_q\otimes_G id_A)_*$ and the action of $Q$ on 
$H_*(G, A)$ at the chain level is the diagonal action on 
$B_*(G)\otimes_G A$. Since the action of $G$ on $A$ is trivial
$$B_*(G)\otimes_G A\cong C_*(G)\otimes A$$
and the action turns out to be the diagonal action on $C_*(G)\otimes A$
which is the same action used in the definition of the homology of invariant group
chains (\ref{accion-sobre-c(g)}).

The next corollary immediately follows from the previous section.

\begin{cor}\label{norma-iso-H}
 If $\Q$ is invertible in $A$, then the norm map $H_*(G, A)_Q\f H_*(G, A)^Q$ is an isomorphism.
\end{cor}
\begin{proof}
 By Proposition \ref{inv-en-H}a $\Q$ is invertible in $H_*(G, A)$ and the claim
 follows from Proposition \ref{norma-iso}b. $\hfill\square$
\end{proof}

\begin{theorem}
 If $\Q$ is invertible in $A$, then the $E^2_{0q}$ term of the spectral sequence
 (\ref{suc-esp}) is isomorphic to $H_q^Q(G,A)$.
\end{theorem}
\begin{proof}
 We have that 
  $$E^2_{0q}=H_0(Q, H_q(G, A))=H_q(G, A)_Q.$$
  By Corollary \ref{norma-iso-H}, $H_q(G, A)_Q\cong H_q(G, A)^Q$, and by Theorem
  \ref{puntos-fijos-H}, $H_q(G, A)^Q\cong H_q^Q(G, A).$ 
 
 $\hfill\square$
\end{proof}

\textbf{Proof of Theorem 2}. Invertibility of $\Q$   in $A$ implies that $\Q$ is  invertible in $H_*(G, A)$ by
  Proposition \ref{inv-en-H}a. By Corollary III.10.2 in \cite{brown} it follows that $H_p(Q, H_q(G, A))=0$ for all $p>0$, thus
  $E^2_{pq}=H_p(Q, H_q(G, A))=0$ if $p\geq1$. This means that the spectral sequence collapses at page two.
 
 Since this sequence converges to $H_*(G\rtimes_\varphi Q, A)$ and according to  the previous theorem 
  $E^2_{0q}\cong H_q^Q(G, A)$, then $H_q^Q(G, A) \cong H_q(G\rtimes_\varphi Q, A).$ $\hfill\square$

\textbf{Example.}
We will compute the groups $H_*^Q(G, A)$ using the preceding Theorem for 
$Q=\Z/2=\langle t\rangle$, $G=\Z/n$ and $A=\Z/m$ with $m$ odd where 
$n\geq 2$, $t(g)=g^{-1}$,  $g\in G$, and the action of $G$ and $Q$
on $A$ is trivial. Under these conditions $G\rtimes_\varphi Q=D_{2n}$
and we just need to compute the homology groups of the dihedral groups.
We will split the computations into two cases depending on the parity of $n$.

Case 1. $n$ odd.

For this case we have that the cohomology groups $H^*(D_{2n}, \Z)$
are given by (\cite{Han93} Theorem 5.3):

$$H^q(D_{2n}, \Z)=\left\{
 \begin{array}{ll}
	\Z& q=0\\
	0 &q\equiv 1\,mod\,4\\
	\Z/2&q\equiv 2\,mod \,4\\
	0&q\equiv 3\, mod \,4\\
	\Z/n\oplus\Z/2& q\equiv 0\,mod\,4, \,\, q>0.
 \end{array}\right.
$$
Using the Universal Coefficient Theorem for cohomology we find that the homology
groups are
\begin{equation}\label{homologyd2nodd}
H_q(D_{2n}, \Z)=\left\{
 \begin{array}{ll}
	\Z& q=0\\
	\Z/2 &q\equiv 1\,mod\,4\\
	0&q\equiv 2\,mod \,4\\
	\Z/n\oplus\Z/2&q\equiv 3\, mod \,4\\
	0& q\equiv 0\,mod\,4, \,\, q>0
 \end{array}\right.
\end{equation}
and using again the Universal Coefficient Theorem now for homology we obtain 
the homology groups with coefficients in any abelian group $B$
$$H_q(D_{2n}, B)=\left\{
 \begin{array}{ll}
	B& q=0\\
	B/2B &q\equiv 1\,mod\,4\\
	T_2(B)&q\equiv 2\,mod \,4\\
	B/2nB& q\equiv 3\, mod \,4\\
	T_{2n}(B)& q\equiv 0\,mod\,4, \,\, q>0
 \end{array}\right.
$$
where $T_n(B)$ denotes the set of all $n$-torsion elements. Using this for 
$B=\Z/m=A$ where $m$ is odd we obtain
$$H_q(D_{2n}, A)=\left\{
 \begin{array}{ll}
	\Z/m& q=0\\
	0 &q\equiv 1\,mod\,4\\
	0&q\equiv 2\,mod \,4\\
	(\Z/m)/2n(\Z/m)=(\Z/m)/n(\Z/m)& q\equiv 3\, mod \,4\\
	T_{2n}(\Z/m)& q\equiv 0\,mod\,4, \,\, q>0.
 \end{array}\right.
$$
But $(\Z/m)/n(\Z/m)\cong \Z/(m,n)$ since
$(\Z/m)/n(\Z/m)\f \Z/(m,n)$ such that 
$[k]\mapsto \bar{k}$
and 
$\Z/(m,n)\f (\Z/m)/n(\Z/m)$
such that $\bar{k}\mapsto [k]$
are well defined group homomorphisms and they are inverses of each other.
On the other hand, it can be seen that 
$$\{\overline{k\tfrac{m}{(m,n)}}|k=0,1,...,(m,n)-1\}\subseteq T_n(A).$$
Moreover, if $\bar{r}\in T_n(A)$ then there exists $l\in \Z$ such that $nr=lm$,
which implies that $r=\frac{lm}{n}=\frac{lm}{s(m,n)}$ for some $s\in \Z$.
Since $(s, m)=1$ and  $r\in\Z$, $s|l$ and thus
$\bar{r}\in \{\overline{k\tfrac{m}{(m,n)}}|k=0,1,...,(m,n)-1\} $.
Then $$\{\overline{k\tfrac{m}{(m,n)}}|k=0,1,...,(m,n)-1\}\supseteq T_n(A).$$
This implies that 
$$T_n(A)=\{\overline{k\tfrac{m}{(m,n)}}|k=0,1,...,(m,n)-1\}\cong \Z/(m,n). $$
Summarizing  we have
$$H_q(D_{2n}, A)=\left\{
 \begin{array}{ll}
	\Z/m& q=0\\
	\Z/(m,n)& q\equiv 0 \,,3\, mod \,4\,,q>0 \\
	0&\mbox{otherwise}.
 \end{array}\right.
$$

Case 2. $n$ even

For this case we have that cohomology groups $H^*(D_{2n}, \Z)$ are given by
(\cite{Han93} Theorem 5.2):
$$H^q(D_{2n}, \Z)=\left\{
 \begin{array}{ll}
	\Z& q=0\\
	(\Z/2)^{\frac{q-1}{2}}&q\equiv 1\,mod\,4\\
	(\Z/2)^{\frac{q+2}{2}}&q\equiv 2\,mod \,4\\
	(\Z/2)^{\frac{q-1}{2}}& q\equiv 3\, mod \,4\\
	\Z/n\oplus(\Z/2)^{\frac{q}{2}}& q\equiv 0\,mod\,4, \,\, q>0.
 \end{array}\right.
$$
Using the Universal Coefficient Theorem for cohomology we find that the homology
groups $H_q(D_{2n}, \mathbb{Z})$ are:
$$H_q(D_{2n}, \Z)=\left\{
 \begin{array}{ll}
	\Z& q=0\\
	(\Z/2)^{\frac{q+3}{2}}&q\equiv 1\,mod\,4\\
	(\Z/2)^{\frac{q}{2}}&q\equiv 2\,mod \,4\\
	\Z/n\oplus (\Z/2)^{\frac{q+1}{2}}& q\equiv 3\, mod \,4\\
	(\Z/2)^{\frac{q}{2}}& q\equiv 0\,mod\,4, \,\, q>0.
 \end{array}\right.
$$
Using the Universal Coefficient Theorem for homology we have for any abelian group $B$:
$$H_q(D_{2n}, B)=\left\{
 \begin{array}{ll}
	B& q=0\\
	(B/2B)^{\frac{q+3}{2}}\oplus T_2(B)^\frac{q-1}{2} &q\equiv 1\,mod\,4\\
	(B/2B)^{\frac{q}{2}}\oplus T_2(B)^\frac{q+2}{2}&q\equiv 2\,mod \,4\\
	B/nB\oplus(B/2B)^{\frac{q+1}{2}}\oplus  T_2(B)^\frac{q-1}{2}& q\equiv 3\, mod \,4\\
	(B/2B)^{\frac{q}{2}}\oplus T_n(B)\oplus T_2(B)^\frac{q}{2}& q\equiv 0\,mod\,4, \,\, q>0.
 \end{array}\right.
$$
In our case $B=\Z/m=A$ with odd $m$ and we have
$$H_q(D_{2n}, A)=\left\{
 \begin{array}{ll}
	\Z/m& q=0\\
	0 &q\equiv 1\,mod\,4\\
	0&q\equiv 2\,mod \,4\\
	(\Z/m)/n(\Z/m)& q\equiv 3\, mod \,4\\
	T_n(\Z/m)& q\equiv 0\,mod\,4, \,\, q>0.
 \end{array}\right.
$$
Using the calculations made in the previous case we have
$$H_q(D_{2n}, A)=\left\{
 \begin{array}{ll}
	\Z/m& q=0\\
	\Z/(m,n)& q\equiv 0 \,,3\, mod \,4\,,q>0 \\
	0&\mbox{otherwise}.
 \end{array}\right.
$$

Summarizing both cases, for $Q=\Z/2$, $G=\Z/n$, $A=\Z/m$ where $m$ 
is odd, $t(g)=g^{-1}$ for every $g\in G$ and $A$ has trival $G$- and $Q$-actions, we have that the homology groups of invariant group chains are:
$$H^Q_q(G, A)=\left\{
 \begin{array}{ll}
	\Z/m& q=0\\
	\Z/(m,n)& q\equiv 0 \,,3\, mod \,4\,,q>0 \\
	0&\mbox{otherwise}.
 \end{array}\right.
$$
We can compare this result with Theorem 6 in \cite{JLM}.

\begin{remark}
 In general, homology of invariant group chains does not agree with the group homology of the semidirect product.
 For instance, consider $Q=\Z/2=\langle t \rangle$ acting on $G=\Z/n$ with odd $n$ by $t(g)=g^{-1},$ $ g\in G$, and $A=\Z$ 
 with trivial actions of $G$ and $Q$. 
 On one hand we have (\cite{Knudson} Example 5.1)
 $$H_p^{\Z/2}(\Z/n, \Z)=\left\{
 \begin{array}{ll}
	\Z& p=0\\
	\Z/n& p\equiv 3\, mod \,4 \\
	0&\mbox{otherwise}.
 \end{array}\right.
$$
On the other hand, the computations made in equation \ref{homologyd2nodd} show us that $H_*^{\Z/2}(\Z/n, \Z)\neq H_*(\Z/n\rtimes \Z/2, \Z)$.
\end{remark}

\section{The first homology group of invariant group chains}


In this section we discuss the relations between Knudson first homology group and two types of abelianization, the first one is and adhoc construction made to coincide with this group, and the second is a natural construction given in terms of the semidirect product of groups and suitable commutants.

\subsection{Orbit group}
Let $Q$ and $G$ be groups, and $Q\times G \to G $ an action of $Q$ on $G$ by 
group automorphisms, i.e.
$$q(g_1g_2)= q(g_1)q(g_2),\qquad q(1)= 1$$
 for every $ q\in Q, g_1,g_2 \in G$.
 
Now, we define a new orbit set. For an element, $g\in G$ consider its 
$Q$-\textit{ordered orbit}:
\begin{align}
(q(g))_{q\in Q}.
\end{align} 
These orbits form a set $\mathcal{O}(G, Q)\subset \prod_{q\in Q}(G)_{q}$, 
\begin{align}
\mathcal{O}(G, Q)&=\left\{(q(g))_{q\in Q}|\; g\in G\right\}. 
\end{align} 
 
The group $Q$ acts on $\mathcal{O}(G, Q)$ by the rule
$$Q\times \mathcal{O}(G, Q)\to \mathcal{O}(G, Q);$$
$$q_{1}\cdot(q(g))_{q\in Q}=(q(q^{-1}_{1}(g)))_{q\in Q}.$$
The resulting orbit set $\mathcal{O}'(G, Q)=\mathcal{O}(G, Q)/Q$ 
is in one-to-one correspondence with the traditional orbit set  $G/Q$, such 
correspondence is induced by the map
$$
\mathcal{O}(G, Q)\longrightarrow G/Q;
$$
$$
(q(g))_{q\in Q}\mapsto \{q(g)\,: \,q\in Q\}.
$$ 
The coordinates of every element in $\mathcal{O}'(G, Q)$ depends 
only on the elements  $q(g)\in G$ for a set of representatives of the classes
$[q]\in Q/Q_{g}$, where $Q_{g}$ denotes the isotropy group of the element $g\in G$.
So, one can denote them as
$$(q(g))_{[q]\in Q/Q_{g}}\in\mathcal{O}'(G, Q)$$
and say that the classes are \textit{ordered up to the action of the group} $Q$. 

Consider the product rule
$$\mathcal{O}(G, Q)\times \mathcal{O}(G, Q) \longrightarrow \mathcal{O}(G, Q)$$
defined by multiplication coordinate by coordinate, i.e.
\begin{align}
(q(g_{1}))_{q\in Q}(q(g_{2}))_{q\in Q} =(q(g_{1})q(g_{2}))_{q\in Q}.
\end{align} This product is well defined, associative, it has the neutral element
\begin{align}
(q(1))_{q\in Q}=(1)_{q\in Q}
\end{align}and the inverse of the element $(q(g))_{q\in Q}\in \mathcal{O}(G, Q)$ 
is the element
\begin{align}
(q(g^{-1}))_{q\in Q}=(q(g)^{-1})_{q\in Q}.
\end{align}Indeed, 
\begin{align*}
(q(g))_{q\in Q}(q(g)^{-1})_{q\in Q}=&(q(g)q(g)^{-1})_{q\in Q}=(1)_{q\in Q}.
\end{align*}

In fact, the group $\mathcal{O}(G, Q)$ is isomorphic to the group $G$. This new expression of the same object is a way to take into account the action of the group $Q$ and to have a well-defined product of $Q$-orbits as a whole.

If all isotropy groups of elements $g\in G$ are finite, then the coordinates of an element 
$(q(g))_{q\in Q}\in \mathcal{O}(G, Q)$  repeat as many times as the order of any isotropy group $Q_{q(g)}$. In the case the whole group $Q$ is finite, one may define the map
\begin{equation}\label{orbitmap}
\mathcal{O}(G, Q) \longrightarrow C_1(G); (q(g))_{q\in Q} \mapsto \sum_{q\in Q}q(g)=|Q_{g}|\sum_{[q]\in Q/Q_{g}}q(g)
\end{equation} 
to the free abelian group generated by $G$ as a set. The image of this map lies in the fixed point subgroup $C_1(G)^Q$ but does not generate the whole group, which is generated by the elements of the form $\sum_{[q]\in Q/Q_{g}}q(g)$.

One may also define the map 
$$\mathcal{O}'(G, Q) \longrightarrow C_1(G); (q(g))_{[q]\in Q/Q_{g}} \mapsto \sum_{[q]\in Q/Q_{g}}q(g).$$ Then one obtains a map $\mathcal{O}(G, Q) \longrightarrow C_1(G)^Q$ as the composition 
$$\mathcal{O}(G, Q) \longrightarrow \mathcal{O}'(G, Q)\longrightarrow C_1(G)^Q$$
whose image generates $C_1(G)^Q$. However, the set $\mathcal{O}'(G, Q)$ does not have a group structure that could make this into a group homomorphism.

One may ask instead for a minimal relation subgroup $R(G,Q)$ in $C_1(G)^Q$ making the map (\ref{orbitmap}) into a group homomorphism, i.e. such that the composition
$$\mathcal{O}(G, Q) \longrightarrow C_1(G)^Q \longrightarrow C_1(G)^Q/R(G,Q)$$
is a homomorphism.

For this, note that the map (\ref{orbitmap}) sends the product $g_1g_2$ to 
$$\mid Q_{g_1g_2}\mid\sum_{[q]\in Q/Q_{g_1g_2}}q[g_1g_2].$$
So, the element  
\begin{equation}\label{isotropyrelation}
\mid Q_{g_1g_2}\mid\sum_{[q]\in Q/Q_{g_1g_2}}q[g_1g_2]-\mid Q_{g_1}\mid\sum_{[q]\in Q/Q_{g_1}}q[g_1]-
\mid Q_{g_2}\mid\sum_{[q]\in Q/Q_{g_2}}q[g_2]
\end{equation}
must belong to $R(G,Q)$. However, there are different pairs of elements $g_1$ and $g_2$ (having different orbits) that may give the same product $g_3=g_1g_2$. Therefore, the relation (\ref{isotropyrelation}) might not be minimal. Note that the integers $\mid Q_{g_1g_2}\mid, \mid Q_{g_1}\mid, \mid Q_{g_2}\mid$ have a common divisor, namely $|Q_{g_1}\cap Q_{g_2}|$. Therefore, the relation (\ref{isotropyrelation}) 
is generated by 
\begin{equation}\label{minimalisotropyrelation}
\begin{array}{l}
\frac{\mid Q_{g_1g_2}\mid}{\mid Q_{g_1}\cap Q_{g_2}\mid}\sum_{[q]\in Q/Q_{g_1g_2}}q[g_1g_2]-\\
-\frac{\mid Q_{g_1}\mid}{\mid Q_{g_1}\cap Q_{g_2}\mid}\sum_{[q]\in Q/Q_{g_1}}q[g_1]-
\frac{\mid Q_{g_2}\mid}{\mid Q_{g_1}\cap Q_{g_2}\mid}\sum_{[q]\in Q/Q_{g_2}}q[g_2].
\end{array}
\end{equation}This justifies the following definition.

\begin{definition}\label{weighedabelianization}
The weighed orbit abelianization $\mathcal{O}(G, Q)_{\text{wab}}$ of the action of a (finite) group $Q$ on a group $G$ by group automorphisms is the biggest abelian group generated by the $Q$-orbits $\sum_{q\in Q}q(g), g\in G$, such that the map 
\begin{align}\label{repeatedabelianprojection}
\mathcal{O}(G, Q) \longrightarrow \mathcal{O}(G, Q)_{\text{wab}};\;(q(g))_{q\in Q}\mapsto \sum_{q\in Q}q(g)
\end{align}
is a group homomorphism.
\end{definition}
The resulting group, which is by definition a quotient of $C_1(G)^Q $ is called weighed because of the relations given in terms of the orders of isotropy groups $Q_{g}$.

By construction, one has the following.

\begin{theorem}\label{firsthomology}
Let $Q$ be a finite group, $G$ be a group, and $Q\times G \to G $ be an action of $Q$ on $G$ by 
group automorphisms. Denote by $H_{1}^{Q}(G,\mathbb{Z})$ the first homology group of $Q$-invariant chains on the group $G$ (see \cite{Knudson}).
Then, there is an isomorphism
\begin{align}\label{1computation}
H_{1}^{Q}(G,\mathbb{Z})\cong \mathcal{O}(G, Q)_{\text{wab}}.
\end{align}
\end{theorem}
\proof\,

It was shown in \cite[Theorem 2]{JLM} the following.

\begin{eqnarray*}
&&H^{Q}_{1}(G,\mathbb{Z})=\\
&&\frac{\mathbb{Z}\left\lbrace
\sum_{\bar{q}\in Q/Q_{[g]}}q[g]\,\vline\, g\in G  \right\rbrace}{\mathbb{Z}\left\lbrace
a\sum_{\bar{q}\in Q/Q_{[g_{2}]}}q[g_{2}]-b\sum_{\bar{q}\in Q/Q_{[g_{1}g_{2}]}}q[g_{1}g_{2}]+c\sum_{\bar{q}\in Q/Q_{[g_{1}]}}q[g_{1}]\,\vline\, g_{1},g_{2}\in G  \right\rbrace},
\end{eqnarray*}
where $a=\frac{\mid Q_{[g_{2}]}\mid}{\mid Q_{[g_{1}]}\cap Q_{[g_{2}]}\mid}$, $b=\frac{\mid Q_{[g_{1}g_{2}]}\mid}{\mid Q_{[g_{1}]}\cap Q_{[g_{2}]}\mid}$ and $c=\frac{\mid Q_{[g_{1}]}\mid}{\mid Q_{[g_{1}]}\cap Q_{[g_{2}]}\mid}$.

$\hfill\square$

\subsection{Relations between $G$ and $H^{Q}_{1}(G,\mathbb{Z})$}

In this section we compare the group $H^{Q}_{1}(G,\mathbb{Z})$ with a 
construction of an abelian group made of orbits of the action of $Q$ on $G$. This construction is a natural extension of the usual notion of abelianization of a group, which is made of orbits of the group $G$ acting on itself by group conjugation.

The map $G\to C_1(G)^Q$ sending each element $g\in G$ to the corresponding generator $\sum_{[q]\in Q/Q_{g}}q[g]$ does not induce a homomorphism between the groups $G$ and $H^{Q}_{1}(G,\mathbb{Z})$. Instead, the (not necessarily surjective) norm map $N:G\to C_1(G)$ given by sending each element $g\in G$ to the sum of the elements on its orbit  $\sum_{[q]\in Q}q[g]=\mid Q_{g}\mid\sum_{[q]\in Q/Q_{g}}q[g]$ does induce a homomorphism, because the product
$g_1g_2$ is sent to 
$$\mid Q_{g_1g_2}\mid\sum_{[q]\in Q/Q_{g_1g_2}}q[g_1g_2]=$$
$$=\mid Q_{g_1}\mid\sum_{[q]\in Q/Q_{g_1}}q[g_1]+
\mid Q_{g_2}\mid\sum_{[q]\in Q/Q_{g_2}}q[g_2].$$
However, this map sends each of the elements $q(g), q\in Q$ of the orbit of $g$ to the same element in $H^{Q}_{1}(G,\mathbb{Z})$. So, one might seek for a group defined in terms of $G$, $Q$ and the action $Q\times G \to G$ identifying all the elements in the same orbit and being abelian.

A version of this in the case of inner automorphisms is the well known abelianization $G_{\text{ab}}$ of the group $G$. This is because this group is generated by conjugacy classes of its elements.   

In the semidirect product $G \rtimes Q$ induced by the action $\phi: Q\times G \to G$ of $Q$ on $G$, consider the commutator subgroup $[G,Q]$. The generators of this subgroup have the form
$[g,q]=g(qg^{-1}q^{-1})=g \phi(q,g^{-1})=g q(g^{-1})$. Denote by $[G,Q]^{G \rtimes Q}$ the normal closure of  $[G,Q]$ in $G \rtimes Q$. Then, in the quotient 
$$
G \rtimes Q/[G,Q]^{G \rtimes Q}
$$one has that $\overline g = \overline{q(g)}$ for every $q\in Q, g\in G$. 

\begin{definition}\label{orbitgroup}
The orbit group of the action of the group $Q$ on $G$, denoted by $G//Q$,
is the image of $G$ in  $(G \rtimes Q)/[G,Q]^{G \rtimes Q}$ under the composition 
\begin{equation}\label{orbitinclusion}
G\hookrightarrow G \rtimes Q \longrightarrow (G \rtimes Q)/[G,Q]^{G \rtimes Q}.
\end{equation}
\end{definition}
Denote by $p:G \to G//Q$ the map (\ref{orbitinclusion}) onto its image.

This group has the following universal property: if $\varphi: G \to H$ is a group homomorphism such that $\varphi (q(g))=\varphi (g)$ for any $q\in Q$ and $g\in G$, then there is a unique homomorphism
$\psi: G//Q \to H$ such that the following diagram is commutative:
$$
\xymatrix{G  \ar[dr]^-{\varphi} \ar[d]_-{p} & \\ 
                G//Q \ar[r]^-{\psi} &    H,} 
$$
i.e.
$\psi \circ p = \varphi$. 

This means that there is an induced map $G//Q \to H^{Q}_{1}(G,\mathbb{Z})$ 
commuting with the map
$\overline N: G \to H^{Q}_{1}(G,\mathbb{Z})$. As $H^{Q}_{1}(G,\mathbb{Z})$ is an abelian group, this factors through a homomorphism $(G//Q)_{\text{ab}} \to H^{Q}_{1}(G,\mathbb{Z})$ sending the element $\bar g\in (G//Q)_{\text{ab}} $ to the element $\mid Q_g\mid [g]^Q\in H^{Q}_{1}(G,\mathbb{Z})$.


By the previous discussion, we have the following.
\begin{theorem}
There is a homomorphism 
\begin{equation}
(G//Q)_{\text{ab}}\longrightarrow H^{Q}_{1}(G,\mathbb{Z})
\end{equation}
from the orbit group abelianization to the first homology group of invariant group chains commuting with the norm map, i.e. such that the diagram
$$\xymatrix{G  \ar[ddr]^-{\bar N} \ar[d]_-{p} & \\ 
                G//Q \ar[d] &   \\
                 (G//Q)_{\text{ab}} \ar[r] & H^{Q}_{1}(G,\mathbb{Z} )} ,$$
is commutative.
\end{theorem}
$\hfill\square$

As it was discussed also, this homomorphism is not in general surjective, because it is induced by the (non surjective) norm map. It might not be injective also: using the relation (\ref{minimalisotropyrelation}), it is easy to show that the order $n$ of an element $g\in G$ annihilates the corresponding element $[g]^Q\in  H^{Q}_{1}(G,\mathbb{Z} )$. Therefore, if $n$ divides  $\mid Q_g\mid$, then the image of such element would be zero.

\section{Invariant cohomology, invariant group extensions with abelian kernel and free actions}

In this section we propose a new definition of invariant cohomology that generalizes the usual cohomology of a group, this cohomology is an invariant of the $Q$-group $G$ that provides algebraic information. We also define their corresponding homology that turns out to be a generalization of the homology defined in \cite{Knudson}. For this, we introduce the category $Q$-$G$ $\mathcal{M}od$ that turns out to be equivalent to the category of modules over $\Z(G\rtimes Q)$.

\begin{definition}
 Let $G$ be a $Q$-group. Let $M$ be, simultaneously, a $G$-module and a $Q$-module. 
 We say that $M$ is a $Q$-$G$ module if 
  \begin{align}
  q(gm)=q(g)qm
  \end{align} for $g\in G, q\in Q, m\in M$. We denote by $Q$-$G\,\mathcal{M}od$ the category whose objects are $Q$-$G$ modules and morphisms are the functions $f:M\rightarrow N$ such that $f$ is both $G$-linear and $Q$-linear.
 \end{definition}
 
\begin{proposition}\label{equivalent}
The category $Q$-$G\,\mathcal{M}od$ is equivalent to the category of modules over $\mathbb{Z}(G\rtimes Q)$.
\end{proposition}
\begin{proof}
Let $T$ be the functor $T:Q\mbox{-}G\,\mathcal{M}od\rightarrow G\rtimes Q\,\mathcal{M}od$, $M\mapsto M'$, $f\mapsto f'$ where $M$ is a $Q$-$G$ module, $f$ is a morphism of $Q$-$G$ modules and $M'$ and $f'$ are defined as follows: as a set $M'$ is defined as the underlying set of $M$ and the action of an element $(g,q)\in G\rtimes Q$ over $m\in M$ is given by $(g,q)m=g(qm)$ and $f'(m)=f(m)$. It is easy to see that $f'$ is a morphism in the category $G\rtimes Q\,\mathcal{M}od$. \\\\ Let $L$ be the functor $L:G\rtimes Q\,\mathcal{M}od\rightarrow Q\mbox{-}G\,\mathcal{M}od$, $M\mapsto \bar{M}$, $f\mapsto \bar{f}$ where $M$ is a $G\rtimes Q$ module and $f$ is a morphism of $G\rtimes Q$ modules. As a set $\bar{M}$ is defined as the underlying set of $M$, an element $q\in Q$ acts on $m\in M$ by the rule $qm=(1,q)m$ and $g\in G$ acts by $gm=(g,1)m$. The morphism $\bar{f}$ is defined by $\bar{f}(m)=f(m)$ for all $m\in M$. It is easy to see that $\bar{f}$ is a morphism in the category $Q\mbox{-}G\,\mathcal{M}od$. \\\\ Thus, $T$ defines an isomorphism of categories with inverse $L$.

 \end{proof}
$\hfill\square$
 

In this way, one says that a $Q$-$G$-module $M$ is a free $Q$-$G$-module, if and only if $T(M)$ is a free $G\rtimes Q$-module. The following proposition is a characterization of such $Q$-$G$ modules:
 
\begin{proposition} A  $Q$-$G$ module $M$ is free if and only if $M$ admits a $\mathbb{Z}G$-basis where $Q$ acts freely.\end{proposition}
 
 \begin{proof} If $M$ admits a $\mathbb{Z}G$-basis where $Q$ acts freely, then one can write $M$ as the sum $M\cong \bigoplus_{i\in I}(\Z G)_i$ with $Q$ acting freely on the set of 
 indices $I$. For each $i\in I$, we have a $\Z(G\rtimes Q)$-isomorphism: 
 $$\bigoplus_{q\in Q}(\Z G)_{qi}\cong \Z(G\rtimes Q)$$ mapping the element $g \in (\Z G)_{qi}$ 
 to the element $(g,q)\in \Z(G\rtimes Q)$. 
 If $E$ is a set of representatives of the quotient $I/Q$ then one has $$M\cong \bigoplus_{j\in E}\bigoplus_{q\in Q}(\Z G)_{qj}\cong \bigoplus_{j\in E} \Z(G\rtimes Q)_j.$$ Conversely, if $M$ is a free $Q$-$G$ module, then $M$ is a free $\Z(G\rtimes Q)$-module, so $$M\cong \bigoplus_{i\in I}\Z(G\rtimes Q)_i\cong \bigoplus_{(q,i)\in Q\times I}(\Z G)_{(q,i)}$$ and the action of $Q$ on $Q\times I$ given by $q'(q,i)=(q'q,i)$ is free.  
\end{proof}

$\hfill\square$


Let $G$ be a $Q$-group and $M$ a $Q$-$G$ module. There are natural actions: $$Q\times Hom_{G}(B_{n}(G),M)\rightarrow Hom_{G}(B_{n}(G),M)$$ $$q\cdot f ([q_{1}\mid\dots\mid g_{n}])=qf([q^{-1}g_{1}\mid \cdots \mid q^{-1}g_{n}])$$ $$Q\times B_{n}(G)\otimes_{G}M\rightarrow B_{n}(G)\otimes_{G}M$$ $$q([q_{1}\mid\cdots\mid g_{n}]\otimes m)=[qg_{1}\mid\cdots\mid qg_{n}]\otimes qm$$ and the differential induced in $Hom_{G}(B(G),M)$ and $B(G)\otimes_{G}M$ by the bar resolution are $Q$-equivariant. We define the homology and cohomology of invariants as: 
\begin{align}
HH_{n}^{Q}(G,M)=&H_{n}((B(G)\otimes_{G}M)^{Q}) \\ 
HH^{n}_{Q}(G,M)=&H^{n}(Hom_{G}(B(G),M)^{Q})
\end{align} 

\begin{remark}
If $Q$ acts trivially on $G$ and on $M$, then $HH^{n}_{Q}(G,M)=H^{n}(G,M)$ and $HH_{n}^{Q}(G.M)=H_{n}(G,M)$. In this way, these invariants are an immediate generalization of the usual homology and cohomology of the group $G$.
\end{remark}

The following propositions show that under certain conditions, this cohomology coincides with other invariants. 

\begin{proposition}
If $Q$ is a finite group and $|Q|$ is invertible in $M$, then $HH^{n}_{Q}(G,M)\cong H^{n}(G,M)^{Q}$.
\end{proposition} 

\begin{proof}
The action of $Q$ on $Hom_{G}(B(G),M)$ induces a well-defined action of $Q$ on $H^{n}(G,M)$ $$Q\times H^{n}(G,M)\rightarrow H^{n}(G,M),\,\,\,(q,\overline f)\mapsto \overline{q\cdot f}$$ The natural inclusion $i:Hom_{G}(B_{n}(G),M)^{Q}\rightarrow Hom_{G}(B_{n}(G),M)$ induces a homomorphism: $$HH^{n}_{Q}(G,M)\rightarrow H^{n}(G,M)^{Q}$$ with inverse given by: $$H^{n}(G,M)^{Q}\rightarrow HH^{n}_{Q}(G,M),\,\,\,\overline{f}\mapsto \overline{\frac{1}{|Q|}\sum_{q\in Q}q\cdot f}$$
\end{proof}

$\hfill\square$

\begin{proposition}
If $Q$ is a finite group, $M$ is a trivial $Q$-$G$ module and $|Q|$ is invertible in $M$, then $ HH^{n}_{Q}(G,M)\cong H^{n}_{Q}(G,M)$.
\end{proposition}

\begin{proof}
It is easy to see that if $M$ is $Q$-$G$ trivial, then $$HH^{n}_{Q}(G,M)\cong H^{n}(Hom(C(G),M)^{Q})$$ We define $$\alpha:Hom(C_{n}(G)^{Q},M)\rightarrow Hom(C_{n}(G),M)^{Q}$$
by 
$$\alpha(f)([g_{1}\mid\cdots\mid g_{n}]\otimes n)=\frac{1}{|Q|}f(\sum_{q\in Q}q[g_{1}\mid\cdots\mid g_{n}]\otimes n).$$
 This homomorphism is an isomorphism with inverse 
 $$
 \beta:Hom(C_{n}(G),M)^{Q}\rightarrow Hom(C_{n}(G)^{Q},M),\,\,\,\beta(f)=f\mid_{C_{n}(G)^{Q}}
 $$ 
 Also, one has $$\alpha(\delta(f))([g_{1}\mid\cdots\mid g_{n+1}]\otimes n)=\frac{1}{|Q|}\delta f(\sum_{q\in Q}q[g_{1}\mid\cdots\mid g_{n+1}]\otimes n)$$ $$=\frac{1}{|Q|}f(\sum_{q\in Q}([qg_{2}\mid\cdots\mid qg_{n}]\otimes n +\sum (-1)^{i}[qg_{1}\mid\cdots\mid qg_{i}qg_{i+1}\mid\cdots\mid qg_{n+1}]\otimes n$$  $$+(-1)^{n+1}[qg_{1}\mid\cdots\mid qg_{n}]\otimes n))$$ $$=\alpha(f)([g_{2}\mid\cdots\mid g_{n+1}]\otimes n)+ \sum_{i\leq n}(-1)^{i}\alpha(f)([g_{1}\mid\cdots\mid g_{i}g_{i+1}\mid \cdots\mid g_{n+1}]\otimes n)$$ $$+(-1)^{n+1}\alpha(f)([g_{1}\mid \cdots\mid g_{n}]\otimes n)=d(\alpha(f))([g_{1}\mid\cdots\mid g_{n+1}]\otimes n)$$ In this way, the diagram: $$\xymatrix{Hom(C_{n}(G)^{Q},M)\ar[r]^{\delta}\ar[d]_{\alpha}&Hom(C_{n+1}(G)^{Q},M)\ar[d]^{\alpha}\\ Hom(C_{n}(G),M)^{Q}\ar[r]_{d}&Hom(C_{n+1}(G),M)^{Q}}$$ is commutative and we obtain an isomorphism: $$HH^{n}_{Q}(G,M)\cong H^{n}(Hom(C(G),M)^{Q})\cong H^{n}_{Q}(G,M)$$
\end{proof}

$\hfill\square$

\subsection{Low dimensional cohomology and group extensions with abelian kernel}
Here we generalize classical results on low dimensional cohomology of groups for the theory we have defined.

In dimension 0, by definition we have $$HH^{0}_{Q}(G,M)=ker(M^{Q}\rightarrow Hom_{G}(B_{1}(G),M)^{Q})=M^{Q}\cap M^{G}$$ but $M^{G}$ admits a $Q$-module structure, so we can write $$HH^{0}_{Q}(G,M)=(M^{G})^{Q}$$ 

Next we describe $HH^{1}_{Q}(G,M)$ in terms of invariant derivations. 

\begin{definition} A $Q$-derivation of $G$ in $M$ is a $Q$-equivariant map $d:G\rightarrow M$ such that $d(g_{1}g_{2})=d(g_{1})+g_{1}d(g_{2})$, we denote the set of $Q$-derivations by $Der_{Q}(G,M)$. For each element $m\in M^{Q}$ we can define an inner $Q$-derivation $d_{m}:G\rightarrow M$, $d_{m}(g)=(g-1)m$, we denote the set of inner $Q$-derivations by $IDer_{Q}(G,M)=\{d_{m}\mid m\in M^{Q}\}$.  
\end{definition} In the sequence: 
$$
\xymatrix{M^{Q}\ar[r]^(.3){\partial_{1}}&Hom_{G}(B_{1}(G),M)^{Q}\ar[r] ^{\partial_{2}}&Hom_{G}(B_{2}(G),M)^{Q}}
$$
$\partial_{1}$ and $\partial_{2}$ are given by: $$\partial_{1}(m)([g])=(g-1)m=d_{m}(g)$$ $$(\partial_{2}f)([g_{1}\mid g_{2}])=g_{1}f([g_{2}])-f([g_{1}g_{2}])+f([g_{1}])$$ so, $Im(\partial_{1})=IDer_{Q}(G,M)$ and $ker(\partial_{2})=Der_{Q}(G,M)$ in this way, $$HH^{1}_{Q}(G,M)=Der_{Q}(G,M)/IDer_{Q}(G,M)$$


Now we discuss group extensions in the present context. In the classical work 
\cite{Eilenberg-Maclane} it is shown
that equivalent classes of group extensions 
\begin{equation}\label{general_extension}
0\rightarrow K \rightarrow G\rightarrow H\rightarrow 0,
\end{equation}
where $K$ is a commutative group and given a (right) action
$H\times K\longrightarrow K$ by group automorphisms, are
classified by the second cohomology group $H^2(H,K)$.

In the sequel, for simplicity, in  (\ref{general_extension})
the group $K$ is identified with its image in $G$ and the group $H$ 
is identified with the quotient $G/K$

For a general group extension  (\ref{general_extension}), as usual, we take a normalized section $s: H \longrightarrow G$,
i.e. $j\circ s = \Id_H $, $s(e)=e$, where $e$ denotes
the neutral element. Now, consider the equation
\begin{equation}\label{multiplication}
x_s\cdot y_s= f(x,y)\cdot (x\cdot y)_s,
\end{equation}where $f(x,y)\in K$, and $x_s$ denotes the image 
of the element $x\in H$ under the section $s$. This gives a so called set of factors 
$f: H\times H \longrightarrow K$. In order to clarify the change in 
equation (\ref{multiplication}) due to a change of representatives in the (right) 
cosets in $G$, we have the following:

For elements $a,b\in K$, the product $(a\cdot x_s)\cdot (b\cdot y_s)$ can be
rewritten in the form
\begin{equation}\label{multiplication_change_representatives}
(a\cdot x_s)\cdot (b\cdot y_s)= a\cdot T_{x}(b) \cdot f(x,y) \cdot(x\cdot y)_s ,
\end{equation}where 
$T_{x}: K\longrightarrow K$ is the (left) inner automorphism given by 
conjugation by the element $x_{s} \in G$, i.e.
\begin{equation}\label{leftconjugation}
T_{x}(a)=x_s a x^{-1}_s,
\end{equation}which does not depend on the representative 
$x_s\in G$ of the class $x\in H$ because the conjugation action 
of the abelian group $K$ on itself is trivial.

Let $G$ be a $Q$-group. Assume that this action leaves the group $K$ invariant, i.e. $QK\subset K$.
 Then $K$ is also a $Q$-group. Moreover, one has
 $q(hx)=q(h)q(x)$, which means that the quotient $H$ is a 
 $Q$-group with the action $Q\times H \to H$ defined 
 by $q(Kg)=Kq(g)$.
 
 We assume that the action of the group $Q$ commutes with the section:  
 $q(x_{s})= (q(x))_{s}$.
 
 One should check how the action of $Q$ on $K$ interacts with its
 $H$-module structure: 
 $$
 q(T_{x}(a))=q(x_s) q(a)q(x_s)^{-1}=(q(x))_s q(a)(q(x))_s^{-1}=T_{q(x)}(q(a)).
 $$If one writes $T_{x}(a)=xa$, then one has $q(xa)=q(x)q(a)$. For such extensions, the factor set $f: H\times H \longrightarrow K$ is 
$Q$-equivariant with respect to the diagonal action $Q\times (H\times H)\to H\times H$.
Indeed, by applying the automorphism $q\in Q$ on both sides of the equation
(\ref{multiplication}) one obtains
\begin{equation}\label{qafuera}
q(x_s)\cdot q(y_s)= qf(x,y)\cdot q((x\cdot y)_s).
\end{equation}Then, applying the same equation to 
$(q(x))_{s}$ and to $(q(y))_{s}$ one has
\begin{equation}\label{qadentro}
q(x)_s\cdot q(y)_s= f(q(x),q(y))\cdot (q(x)\cdot q(y))_s.
\end{equation}But, by assumption, $q(x_{s})= (q(x))_{s}$, 
$q(y_{s})= (q(y))_{s}$.
So, the right side of equations (\ref{qafuera}) and (\ref{qadentro})
coincide and one can cancel the factor 
$$q((x\cdot y)_s)=(q(x\cdot y))_s= (q(x)\cdot q(y))_s$$ on the left side
of these equations.\\

In this way, we have the following:

\begin{theorem}\label{H2}
Let $G$, $H$ be $Q$-groups and let $K$ be a $Q$-$H$ module. Then the set of equivalence classes of $Q$-equivariant extensions 
\begin{equation}
0\rightarrow K \rightarrow G\rightarrow H\rightarrow 1,
\end{equation} that admit a normalized $Q$-equivariant section $s:H\rightarrow G$ is in one-to-one correspondence with 
the elements of the group $HH^{2}_{Q}(H,K)$.
\end{theorem}
\proof\; The previous arguments show that such extensions with $Q$-linear sections 
are defined by $Q$-linear set of factors. 
The usual arguments for classifying group extensions
follow. $\hfill \square$

It is clear that a $Q$-linear set of factors restricts to a homomorphism on fixed points, because 
the subgroup $ C_{\bullet}(H)^{Q}\subset  C_{\bullet}(H)$ is $Q$-invariant.
However, it is not true in general that every homomorphism 
$ f\in\mathrm{Hom}(C_{\bullet}(H)^{Q}, K)$ is the restriction of some 
 $\tilde f\in\mathrm{Hom}_{Q}(C_{\bullet}(H), K)$. This is because, for elements $x\in C_{\bullet}(H)^{Q}$one has  $q\tilde f(x)=\tilde f(qx)=\tilde f(x)$, and this means that
the image of $C_{\bullet}(H)^{Q}$ under $\tilde f$ is contained in the subgroup $K^{Q}$ of invariants of the $Q$-module $K$, i.e. one has a map
$$
\tilde f\mid_{C_{\bullet}(H)^{Q}}:C_{\bullet}(H)^{Q}\longrightarrow K^{Q}.
$$ The set of all such restrictions may not coincide with $\mathrm{Hom}(C_{\bullet}(H)^{Q}, K)$ if $K$ is not a trivial $Q$-module.

\begin{cor}
If $M$ is a $Q$-$G$ module with trivial actions and $|Q|$ is invertible in $M$, then the group $H^{2}_{Q}(G,M)$ classifies $Q$-equivariant extensions $$0\rightarrow M\rightarrow E\rightarrow G\rightarrow 1$$ inducing a trivial action of $G$ on $M$.
\end{cor}

\subsection{Free actions}
One of the greatest difficulties in the study of invariant cohomology lies in its very definition, since the bar resolution $B(G)\rightarrow \Z$ is not a projective resolution in the category $Q$-$G$ $\mathcal{M}od$. In this section, we analyze the case when the action of $Q$ on $G$ is ``free'' in order to replace the bar resolution with a projective resolution of the augmentation ideal. At the end of the paragraph we will analize some examples. 

\begin{definition} An action of a group $Q$ on a group $G$ is free if $Q_{g}=\{1\}$ for each $g\in G, g\neq 1$. In this case, we say that $G$ is a free $Q$-group.
\end{definition}

\begin{remark} $B_{0}(G)=\mathbb{Z}G$ is a free $\mathbb{Z}G$-module where any $\mathbb{Z}G$-basis consists of a single element. In this way, if $|Q|>1$, $Q$ can not act freely on this basis. Therefore, $B_{0}(G)$ is not a free $Q$-$G$ module.
\end{remark}

\begin{proposition}If $G$ is a free $Q$-group, then $B_{n}(G)$ is a free $Q$-$G$ module for each $n>0$. 
\end{proposition}

\begin{proof} The set $\{[g_{1}\mid\cdots\mid g_{n}]\mid g_{i}\in G,g_{i}\neq 1\}$ is a $\Z G$-base and $Q$ acts freely on it. Thus $B_{n}(G)$ is a free $Q$-$G$ module for each $n>0$.
\end{proof}

$\hfill\square$

If $G$ is a free $Q$-group, we can consider the restriction of the bar resolution to the augmentation ideal and we obtain a projective resolution: $$\xymatrix{\cdots\ar[r]&B_{2}(G)\ar[r]^{d_{2}}&B_{1}G\ar[r]^{d_{1}}&I_{G}\ar[r]&0}$$ of the augmentation ideal $I_{G}$ in the category $Q$-$G$ $\mathcal{M}od$. In addition, the natural morphism: $$\eta:Hom_{G}(I_{G},M)\rightarrow Der(G,M),\,\,\,f\mapsto d_{f},$$ where $d_{f}(g)=f(g-1)$, is a $Q$-module isomorphism. 
Then, $$Ext^{0}_{Q\mbox{-}G}(I_{G},M)=Hom_{G}(I_{G},M)^{Q}\cong Der_{Q}(G,M)$$ and, therefore, we obtain the following expression for $HH^{n}_{Q}(G,M)$ by applying the derived functor of $Hom_{Q\mbox{-}G}(\,-\,,M)$ of the $Q$-$G$ module $I_{G}$: 
$$HH^{n}_{Q}(G,M)=\left\{ \begin{array}{lcc} (M^{G})^{Q} & n=0 \\ \\  Ext^{0}_{Q\mbox{-}G}(I_{G},M)/IDer_{Q}(G,M) & n=1 \\ \\Ext^{n-1}_{Q\mbox{-}G}(I_{G},M) & n\geq 2 \end{array}  \right.$$
If $A$ is a $Q$-$G$ module with trivial actions, then $IDer_{Q}(G,A)=0$ and we obtain: $$HH^{n}_{Q}(G,A)=\left\{ \begin{array}{lcc} A & n=0 \\ \\  Ext^{n-1}_{Q\mbox{-}G}(I_{G},M) & n\geq 1 \end{array}  \right.$$

\begin{remark} If the action of $Q$ on $G$ is free and $M$ is a $Q$-$G$ module, then a $Q$-equivariant extension: $$0\rightarrow M\rightarrow E\rightarrow G\rightarrow 1$$ always admits a normalized $Q$-equivariant section $s:G\rightarrow E$. In this case, in theorem \ref{H2}, we can consider all extensions and not only those that have a normalized $Q$-equivariant section. 
\end{remark}

\textbf{Example.} Let $Q$ be a group and let $S$ be a set such that $Q$ acts freely on $S$. Then the free group generated by $S$, $F(S)$ is a $Q$-group with action induced by the action of $Q$ on $S$. Then $I_{F(S)}$ is a free $\mathbb{Z} F(S)$-module with basis $S-1=\{s-1\mid s\in S\}$ in addition, $Q$ acts freely on this basis. So $I_{F(S)}$ is a projective $Q$-$F(S)$ module. Then, $$\xymatrix{0\ar[r]&I_{F(S)}\ar[r]^{1}&I_{F(S)}\ar[r]&0}$$ is a projective resolution of the  $I_{F(S)}$ in $Q$-$F(S)$ $\mathcal{M}od$ therefore, we have: $$HH^{n}_{Q}(F(S),\Z)=\left\{ \begin{array}{lcc} \Z & n=0 \\ \\  Der_{Q}(F(S),\Z) & n=1 \\ \\ 0 & n\geq 2 \end{array}  \right.$$ Since $HH^{2}_{Q}(F(S),\Z)=0$, all the extensions that induce the trivial action on $\Z$ are $Q$-equivalent to the extension of the direct product: $$0\rightarrow \Z\rightarrow \Z\times F(S)\rightarrow F(S)\rightarrow 1$$

\textbf{Example.} Consider the groups $Q=\Z/2=\langle s\mid s^{2}\rangle$ and $G=\Z=\langle t\rangle$ with $\Z/2$ acting freely on $\Z$ by $$s:\Z\rightarrow \Z,\,\,\,t\mapsto t^{-1}$$ in this way, $\Z$ is a free $\Z/2$-group. The augmentation ideal $I_{\Z}$ is the free $\Z G$-module generated by the element $t-1$. It is easy to see that $I_{\mathbb{Z}}$ is not a free $\Z/2$-$\Z$ module. Consider the following projective resolution for that ideal:$$\xymatrix{\cdots\ar[r]&\Z G(\Z/2)\ar[r]^{d_{2}}&\Z G(\Z/2)\ar[r]^{d_{1}}&\Z G(\Z/2)\ar[r]^(.55){\epsilon}&I_{\Z}\ar[r]&0}$$ 
where $\Z G(\Z/2)$ is the free $\Z/2$-$\Z$ module generated by $\Z/2$ and the differentials are given by: 
 $$\epsilon:\Z G(\Z/2)\rightarrow I_{\Z},\,\,\,\,x+ys\mapsto (x-t^{-1}y)(t-1)$$ $$d_{i}(x+ys)=\left\{ \begin{array}{ll}x(1+ts)+yt^{-1}(1+ts),& i=2k-1 \\ \\ x(1-ts)-yt^{-1}(1-ts), & i=2k \end{array}  \right.$$
 with $x,y\in \Z$. When applying the functor $Hom_{G}(\,-\,,\Z)^{Q}$ to the resolution (we are considering the coefficients module $\Z$ as a $\Z/2$-$\Z$ module with trivial actions) we obtain $Hom_{G}(I_{\Z},\Z)^{Q}=Der_{Q}(\Z,\Z)=0$ and $Hom_{G}(\Z G(\Z/2),\Z)^{Q}\cong \Z$. In this way, we have the following cochain complex: 
 $$\xymatrix{0\ar[r]& \Z\ar[r]^{2}&\Z\ar[r]^{0}&\Z\ar[r]^{2}&\Z\ar[r]&\cdots}$$ therefore, 
 $$
 Ext^{n}_{\Z/2\mbox{-}\Z}(I_{\Z},\Z)=\left\{ \begin{array}{ll} 0 & n=2k \\ \\ \Z/2 & n=2k-1 \end{array}  \right.$$ and $$HH^{n}_{\Z/2}(\Z,\Z)=\left\{ \begin{array}{ll} \Z & n=0 \\ \\0&n=2k-1 \\ \\ \Z/2 & n=2k\geq 2 \end{array}  \right.
 $$
 Since $HH^{2}_{\Z/2}(\Z,\Z)=\Z/2$, there are only two $\Z/2$-equivalent classes of extensions: $$0\rightarrow \Z\rightarrow E\rightarrow \Z\rightarrow 1$$ inducing the trivial action on $\Z$.

\section*{Acknowledgements}
The authors express their gratitude to the anonymous referee for the careful revision and useful comments to improve this work.

\end{document}